\documentclass[12pt,reqno]{amsart}

\usepackage{hyperref}

\usepackage{amsmath,amsfonts,amsthm,amssymb,amsxtra,bbm}

\newtheorem{theorem}{Theorem}[section]

\newtheorem{lemma}[theorem]{Lemma}
\newtheorem{corollary}[theorem]{Corollary}

\theoremstyle{definition}

\theoremstyle{remark}

\numberwithin{equation}{section}

\newcommand{\C}{\mathbb{C}}

\renewcommand{\epsilon}{\varepsilon}

\newcommand{\loc}{{\rm loc}}

\newcommand{\N}{\mathbb{N}}

\renewcommand{\phi}{\varphi}
\newcommand{\R}{\mathbb{R}}

\newcommand{\Z}{\mathbb{Z}}

\DeclareMathOperator{\im}{Im}

\DeclareMathOperator{\re}{Re}

\DeclareMathOperator{\supp}{supp}
\DeclareMathOperator{\sgn}{sgn}
\DeclareMathOperator{\Tr}{Tr}
\DeclareMathOperator{\tr}{Tr}

\renewcommand\Im{\hbox{{\rm Im}}\,}
\renewcommand\Re{\hbox{{\rm Re}}\,}
\newcommand{\abs}[1]{\lvert#1\rvert}

\newcommand{\norm}[1]{\lVert#1\rVert}

\newcommand{\jap}[1]{\langle#1\rangle}

\newcommand{\bbR}{{\mathbb R}}
\newcommand{\bbC}{{\mathbb C}}

\newcommand{\Sch}{\mathbf{S}}

\newcommand{\wt}{\widetilde}
\newcommand{\eps}{\varepsilon}

\newcommand{\1}{\mathbbm{1}}
\newcommand{\x}{\mathbf x}
\newcommand{\p}{\mathbf p}

\begin{document}

\title[Trace class conditions]{Trace class conditions for functions of Schr\"odinger operators}

\author{Rupert L. Frank}
\address{Rupert L. Frank, Mathematics 253-37, Caltech, Pasadena, CA 91125, USA}
\email{rlfrank@math.princeton.edu}

\author{Alexander Pushnitski}
\address{Alexander Pushnitski, Department of Mathematics, King's College London, Strand, London, WC2R 2LS, UK}
\email{alexander.pushnitski@kcl.ac.uk}



\begin{abstract}
We consider the difference $f(-\Delta +V)-f(-\Delta)$ of functions of Schr\"odin\-ger operators in $L^2(\R^d)$ and provide conditions under which this difference is trace class. We are particularly interested in non-smooth functions $f$ and in $V$ belonging only to some $L^p$ space. This is motivated by applications in mathematical physics related to Lieb--Thirring inequalities. We show that in the particular case of Schr\"odinger operators the well-known sufficient conditions on $f$, based on a general operator theoretic result due to V.~Peller, can be considerably relaxed. We prove similar theorems for $f(-\Delta +V)-f(-\Delta)-\frac{d}{d\alpha} f(-\Delta +\alpha V)|_{\alpha=0}$. Our key idea is the use of the limiting absorption principle. 
\end{abstract}

\maketitle

\renewcommand{\thefootnote}{${}$} \footnotetext{\copyright\, 2013 by the authors. This paper may be reproduced, in its entirety, for non-commercial purposes.}

\section{Introduction and main results}\label{sec.a}

\subsection{Setting of the problem}\label{sec:setting}

In this paper we consider functions $f(H)$ and $f(H_0)$ of the perturbed and unperturbed Schr\"odinger operators
\begin{equation}
\label{1}
H= -\Delta +V \,,\qquad H_0 = -\Delta
\qquad
\text{in}\ L^2(\R^d)
\end{equation}
and we investigate which assumptions on the real-valued potential $V$ and on the function $f$ guarantee the property that
\begin{equation}
\label{eq:difftcintro}
f(H) - f(H_0) \in \mathbf S_1
\end{equation}
or
\begin{equation}
\label{eq:diff2tcintro}
f(H) - f(H_0) - \frac{d}{d\alpha} f((1-\alpha)H_0 + \alpha H)|_{\alpha=0} \in \mathbf S_1 \,,
\end{equation}
where $\mathbf S_1$ denotes the trace class. The potential $V$ will always be assumed infinitesimally form bounded with respect to $-\Delta$ and to decay (pointwise or in some $L^p$ sense) at infinity. 
We will be more specific below.

If $f$ is smooth, say, $f\in C_0^\infty(\bbR)$, and $V$ decays sufficiently rapidly at infinity, then \eqref{eq:difftcintro} and \eqref{eq:diff2tcintro} are certainly true and this can be proved by 
several standard methods. Here, we are mostly interested in functions $f$, which are absolutely continuous but not much smoother. This makes the question much more subtle.

There are at least two motivations for considering such $f$.
\begin{enumerate}
\item One of us (A. P.) studied the difference $f(H) - f(H_0)$ for functions $f$ with jump discontinuities 
\cite{Pu1,Pu2,PuYa}. Among other things, it was shown that for the function $f(\lambda)=\1_{(-\infty,a)}(\lambda)$ with $a>0$ the operator $f(-\Delta+V) - f(-\Delta)$ is never compact, unless scattering at energy $a$ is trivial. This naturally raises the question how the transition from non-compact to trace class occurs as the smoothness of $f$ increases.
\item One of us (R. F.) proved bounds on `something like' the trace of the left sides of \eqref{eq:difftcintro} and \eqref{eq:diff2tcintro} for the Lipschitz functions $f(\lambda) = (\lambda-a)_-$ with $a>0$ \cite{FLLS}. 
(Here and in what follows, $x_{\pm}=\max\{\pm x,0\}$.) 
The purpose of \cite{FLLS}  was achieved by introducing a certain regularised notion of trace, 
but the question, whether these operators are actually trace class, was left as an open problem. 
Lipschitz functions of this form arise naturally in a problem in mathematical physics related to 
Lieb--Thirring inequalities that we sketch in Subsection \ref{sec:dens}.
\end{enumerate}

One attempt to answer these questions is to look at abstract results in operator theory. The problem of giving sufficient conditions on functions $f$ such that the implications
\begin{equation}
H-H_0\in\Sch_1\quad \Rightarrow \quad f(H)-f(H_0)\in \Sch_1
\label{6}
\end{equation}
or
\begin{equation}
H-H_0\in\Sch_2\quad \Rightarrow \quad f(H)-f(H_0)- \frac{d}{ d\alpha} f((1-\alpha)H_0 + \alpha H)|_{\alpha=0} \in \Sch_1
\label{12}
\end{equation}
hold for an \emph{arbitrary} pair of self-adjoint operators $H$ and $H_0$ was considered in many works including, in particular, \cite{Krein, BS,Farforovskaya,Kato,Peller,Davies}. In \eqref{12}, $\Sch_2$ denotes the Hilbert--Schmidt class. (Of course, in the Schr\"odinger case $H-H_0$ is never compact, but one would like to apply these abstract results to the difference of (powers of) resolvents.) One of the sharpest sufficient conditions for \eqref{6} was obtained by V.~Peller in terms of Besov spaces $B^s_{p,q}(\R)$ whose definition we recall in Section \ref{sec.b}. In \cite{Peller} he showed that 
\begin{equation}
f\in B^1_{\infty,1}(\R) \text{ implies \eqref{6}}; 
\label{12a}
\end{equation}
for a precise statement, see Theorem \ref{thm.p1} below. The condition $f\in B^1_{\infty,1}(\bbR)$ is, roughly speaking, just a little stronger than the requirement $f'\in L^\infty(\bbR)$. Some necessary conditions for \eqref{6} are also known \cite{Peller}; for example, $f$ needs to be continuous, differentiable and satisfy $f'\in L^\infty_\loc(\bbR)$. In terms of the local behaviour of $f$, we get that the functions that behave like $f(\lambda)=(\lambda-a)_\pm^{\gamma}$ near $\lambda=a$ (and are smooth elsewhere) satisfy \eqref{6} if and only if $\gamma>1$. In particular, the function $f(\lambda) = (\lambda-a)_-$ that appears in the above problem in mathematical physics does \emph{not} fit into this abstract framework.


\subsection{Main results}

The main point of this paper is to show that for some 
\emph{particular pairs of operators $H_0$, $H$,} satisfying some 
standard assumptions of smooth and trace class scattering theory,
\emph{the class of admissible functions $f$ for the inclusion $f(H)-f(H_0)\in\Sch_1$ 
is much wider} and includes functions $f$
of the type
\begin{equation}
f_{\gamma,a}(\lambda)=(\lambda-a)^\gamma_-;
\quad a>0, 
\label{5} 
\end{equation}
\emph{for all $\gamma>0$.} 
We shall focus on the particular case of the Schr\"odinger operator, although the 
results could be extended to a much wider setting by using the language
of abstract scattering theory.

We focus on the local behaviour of $f$ on the continuous spectrum of $H_0$, i.e., on $[0,\infty)$. 
The question of the behaviour of $f$ at $+\infty$ and
near zero are of a very different nature, so in the following discussion
we will assume (most of the time) that $f$ is compactly supported on $(0,\infty)$. 

We start with the following preliminary result.

\begin{theorem}\label{thm0}
Let $d\geq 1$ and assume that $V$ satisfies the pointwise bound
\begin{equation}
\abs{V(\x)}\leq C(1+\abs{\x})^{-\rho}, 
\quad \x\in \bbR^d, \quad\text{for some}\  \rho>d,
\label{2}
\end{equation}
and some $C>0$. 
If $f\in B^1_{1,1}(\R)$ has compact support in $(0,\infty)$, then
$$
f(-\Delta+V) - f(-\Delta) \in\mathbf S_1 \,.
$$
This inclusion holds also for the functions \eqref{5} for any $\gamma>0$ and $a>0$.
\end{theorem}

In other words, the assumption $f\in B^1_{\infty,1}(\R)$ from abstract theory (see \eqref{12a}) can be replaced 
by the assumption $f\in B^1_{1,1}(\R)$ which in the model case \eqref{5} lowers the requirement 
on the exponent from $\gamma>1$ to $\gamma>0$.

We note that although the function $f_{\gamma,a}$ in \eqref{5} for $\gamma>0$ is, strictly
speaking, not in the class $B^1_{1,1}(\R)$ because of its growth at minus infinity, one can
easily write it as $f_{\gamma,a}=f_0+f_1$, where $f_0\in B^1_{1,1}(\R)$ and $f_1$ vanishes
on the spectra of the operators $-\Delta+V$ and $-\Delta$.

Further, we are able to replace pointwise condition \eqref{2} by more general 
$L^p$ conditions. 
Our main result is the following.

\begin{theorem}\label{main}
Let $d\geq 1$ and 
\begin{align*}
& V \in L^1(\R) & \text{if}\ d=1 \,, \\
& V \in L^1(\R^2)\cap L^p(\R^2) \ \text{for some}\ p>1 & \text{if}\ d=2\,,\\
& V\in L^1(\R^3)\cap L^{3/2}(\R^3) & \text{if}\ d=3 \,,\\
& V\in \ell^1(L^2) \cap L^{d/2}(\R^d) & \text{if}\ d\geq 4 \,.
\end{align*}
If $f\in B^1_{1,1}(\R)$ has compact support in $(0,\infty)$, then
$$
f(-\Delta+V) - f(-\Delta) \in\mathbf S_1 \,.
$$
This inclusion holds also for the functions \eqref{5} for any $\gamma>0$ and $a>0$.
\end{theorem}

Of course, Theorem~\ref{thm0} follows from Theorem~\ref{main}.
Note that $-\Delta +V$ can be defined via a quadratic form with form domain $H^1(\R^d)$ if $V$ satisfies the assumptions of Theorem~\ref{main}. We also recall that the space $\ell^1(L^2)$ that appears in the above theorem is defined by the requirement that
$$
\sum_{n\in\Z^d} \left( \int_{Q_n} |V|^2 \,dx\right)^{1/2} \,,
\qquad Q_n = n+(-1/2,1/2)^d \,,
$$
is finite. It is easy to see that $\ell^1(L^2)\subset L^1\cap L^2$ and that
$$
\left\{ V: (1+|x|)^{\sigma} V \in L^2 \right\} \subset \ell^1(L^2)
\qquad\text{if}\ \sigma>d/2 \,.
$$

Our second main result concerns the inclusion \eqref{eq:diff2tcintro}. Again there is an abstract result of Peller
\cite{Peller2} (motivated by earlier work of L. Koplienko \cite{Koplienko}) which proves \eqref{12} for $f\in B^2_{\infty,1}(\R)$; see Theorem \ref{thm.p2} below. The requirement $f\in B^2_{\infty,1}(\bbR)$ is, roughly speaking, just a 
little stronger than $f''\in L^\infty(\bbR)$. In particular, it is easy to see that functions $f$ with local singularities $f(\lambda)=(\lambda-a)_\pm^\gamma$ are admissible if and only if $\gamma>2$. Again, it turns out that for Schr\"odinger operators this holds under considerably weaker regularity conditions.
We shall prove

\begin{theorem}\label{thm2}
Let $d=1,2,3$ and $V\in L^2(\R^d)$. If $f\in B^2_{1,1}(\R)$ has compact support in $(0,\infty)$, then
$$
f(-\Delta+V) - f(-\Delta) - \frac{d}{d\alpha} f(-\Delta +\alpha V)|_{\alpha=0} \in\mathbf S_1 \,.
$$
This inclusion holds also for the functions \eqref{5} for any $\gamma>1$ and $a>0$.
\end{theorem}

In order to keep the paper reasonably short and elementary we have proved this only for dimensions $d\leq 3$. We expect that a similar theorem holds in general dimensions.

The method that we introduce in this paper not only allows to prove Theorems \ref{main} and \ref{thm2}, but also provides a short alternative proof of both theorems of Peller that were mentioned before; see Theorems \ref{thm.p1} and \ref{thm.p2} below. We also point out that the same method allows one to obtain analogues of Theorems \ref{main} and \ref{thm2} for Schatten classes $\mathbf S_p$, $p>1$, under slightly different assumptions on $V$.



\subsection{Motivation from mathematical physics}\label{sec:dens}

Many challenging problems in mathematical physics are related to understanding the quantum many-body problem. One of the approaches that has been successfully employed in some limiting regimes is to approximate the Hamiltonian of the many-body system by a one-body Schr\"odinger operator $-\Delta+V$ with an effective potential $V$. If the particles are fermions, the ground state energy is then given (up to spin degeneracies) by the sum of the lowest eigenvalues of $-\Delta+V$. If there is no restriction on the number of particles, this sum is at least $-\tr(-\Delta+V)_-$. The mathematical tool both for estimating the latter quantity and for justifying the approximation by a one-body Schr\"odinger operator is the Lieb--Thirring inequality \cite{LT},
\begin{equation}
\tr (-\Delta+V)_-^\gamma \leq L_{\gamma,d} \int_{\R^d} V_-^{\gamma+d/2}(\x)\,d\x
\label{lt1}
\end{equation}
with $\gamma=1$. Here the constant $L_{\gamma,d}$ is independent of $V$. It is useful and interesting to study the above inequality also for different values of $\gamma$ and we refer to \cite{LW2,H2} for reviews of the field and precise statements.

Above we were assuming that the number of particles is negligible with respect to the size of the system. However, for instance in solids the number of particles is proportional to the volume and in this case the energy of the system is approximated by $-\tr(-\Delta+V-\mu)_-$ for a positive constant $\mu$ (the chemical potential). While $-\tr(-\Delta+V-\mu)_-$ is finite if the Schr\"odinger operator is considered on a bounded domain, a regularization is needed in order to treat the problem on the whole space. Formally, one subtracts $-\tr(-\Delta-\mu)_-$, which is interpreted as the total energy of the background. The question whether analogues of the Lieb--Thirring inequality extend to this situation has been considered only recently in \cite{FLLS}; see also \cite{FLLS2}. While the natural definition of a relative energy is $-\tr\left((-\Delta+V-\mu)_--(-\Delta-\mu)_-\right)$, a regularized definition was used in \cite{FLLS} in order to avoid discussing the trace class properties of $(-\Delta+V-\mu)_--(-\Delta-\mu)_-$.
(This regularization also avoids having $V\in L^1(\R^d)$, although this
will not be important for us here.)

Using our Theorem \ref{main} and a key estimate from \cite{FLLS},
we are able to prove this bound without any regularization. 
We denote
$$
L_{\gamma,d}^{\mathrm{sc}} = \int_{\R^d} (|\p|^2-1)_-^{\gamma} \frac{d\p}{(2\pi)^d}
$$
(``sc'' stands for semiclassical); this is the constant that one expects in \eqref{lt1} from semiclassical phase space considerations. 
\begin{theorem}\label{lt}
Let $d\geq 2$. 
Then there is a constant $L_{1,d}$ such that for all $\mu\in\R$ and all 
$V\in L^1(\R^d)\cap L^{1+d/2}(\R^d)$ (with $V\in \ell^1(L^2)$ if $d\geq 4$)
one has
\begin{align*}
0 & \leq \tr\left( \left(-\Delta+V-\mu\right)_- - \left(-\Delta -\mu \right)_- \right) 
+ L_{0,d}^{\mathrm{sc}} \ \mu_+^{d/2} \int_{\R^d} V(\x) \,d\x \\
& \leq L_{1,d} \int_{\R^d} \left( \left(V(\x)-\mu\right)_-^{1+d/2} - \mu_+^{1+d/2} + 
\left(1+\frac d2\right) \mu_+^{d/2} V(\x) \right) d\x.
\end{align*}
\end{theorem}
Of course, $\mu>0$ is the only novel case; we get the case $\mu\leq0$ immediately
from \eqref{lt1}.
We also obtain the $\gamma>1$ versions of the inequality (they are a simple corollary of the $\gamma=1$ case):

\begin{corollary}\label{lt_cor}
Let $d\geq 2$ and $\gamma>1$. Then there is a constant $L_{\gamma,d}$ such that for all $\mu\in\R$ and all $V\in L^1(\R^d)\cap L^{\gamma+d/2}(\R^d)$ (with $V\in \ell^1(L^2)$ if $d\geq 4$) one has
\begin{align*}
0 & \leq \tr\left( \left(-\Delta+V-\mu\right)_-^\gamma - \left(-\Delta -\mu \right)_-^\gamma \right) 
+ L_{\gamma-1,d}^{\mathrm{sc}} \ \mu_+^{\gamma+d/2-1} \int_{\R^d} V(\x) \,dx 
\\
& \leq L_{\gamma,d} \int_{\R^d} \left( \left(V(\x)-\mu\right)_-^{\gamma+d/2} - \mu_+^{\gamma+d/2} + 
\left(\gamma+\frac d2\right) \mu_+^{\gamma+d/2-1} V(\x) \right) d\x.
\end{align*}
\end{corollary}

Beyond the applications to minimization problems mentioned above, we note that Lieb--Thirring inequalities with $\mu>0$ recently also proved useful in time-dependent problems \cite{LS}.


\subsection{Key ideas of the proof}\label{sec.a4}
First assume that $f\in  C_0^\infty(\bbR)$. 
Then for any given $N\in\N$ one can construct an almost analytic extension $\wt f\in C_0^\infty(\bbC)$ 
with the properties $\wt f|_{\bbR}=f$ and 
$$
\abs{\overline{\partial} \wt f(z)}\leq C_N \abs{\Im z}^N \,;
$$
here 
$\overline{\partial} 
=\frac{\partial}{\partial \overline{z}}
=\frac12(\frac{\partial}{\partial x}+i\frac{\partial}{\partial y})$.
This allows to represent 
$$
f(H)=\frac1\pi \int_\bbC \overline{\partial} \wt f(z) (H-z)^{-1} dx\, dy, \quad z=x+iy
$$
for any self-adjoint operator $H$. 
Based on this idea (which has been rediscovered several times) several versions of functional calculus have been constructed by many authors \cite{Dynkin2, HelfferSjostrand, JensenNakamura, Davies2, DimassiSjostrand}.

Further, this representation can be applied to perturbation theory as follows. 
Let $H_0$ and $H$ be two self-adjoint operators; denote $V=H-H_0$ and 
\begin{equation}
R(z)=(H-z)^{-1}, \quad R_0(z)=(H_0-z)^{-1}. 
\label{23}
\end{equation}
Then, by the resolvent identity, 
\begin{equation}
f(H)-f(H_0)
=
-\frac1\pi \int_\bbC \overline{\partial} \wt f(z)R(z)VR_0(z) \,dx\, dy\,, \quad z=x+iy \,,
\label{24}
\end{equation}
and so one can derive estimates for the norm of $f(H)-f(H_0)$ in appropriate 
classes from the available estimates for the corresponding norms of $R(z)VR_0(z)$. 
This idea has been extensively used before (see, e.g., \cite{DimassiSjostrand} and 
references therein).
Our construction is based on the following two additional observations: 
\begin{enumerate}
\item[(1)] For the Schr\"odinger operator, the available estimates for 
$R(z)VR_0(z)$ are better than one would expect for a general pair 
of operators $H_0$, $H$ under some trace class condition 
(such as $V\in \Sch_1$ or its variants). This is essentially due 
to the \emph{limiting absorption principle}.
\item[(2)] One can go far beyond the class $f\in C_0^\infty(\bbR)$. 
In fact, E.~M.~Dynkin \cite{Dynkin} has a beautiful characterisation 
of Besov classes $B^s_{p,q}(\bbR)$ in terms of the behaviour 
of the almost analytic extension (see Theorem~\ref{thm.dynkin} below).
\end{enumerate}

Combining Dynkin's theorem with available estimates for $R(z)VR_0(z)$
gives surprisingly sharp results in a surprisingly elementary way. 
For example, let us sketch

\begin{proof}[Proof of Theorem~\ref{thm0} for $d=1,2,3$]
Under the assumption \eqref{2} with $\rho>1$, one has the 
standard limiting absorption principle:
\begin{equation}
\sup_{\re z\in\delta, \im z \neq0}\norm{\jap{\x}^{-\rho/2}R(z)\jap{\x}^{-\rho/2}}\leq C_\rho(\delta)\,,
\qquad \rho>1,
\label{10}
\end{equation}
where $\delta\subset(0,\infty)$ is any compact interval;
see, e.g., \cite[Thm. 6.2.1]{Ya}. 
From here, 
using the resolvent identity and a trivial Hilbert-Schmidt bound (using $\rho>d$; see Lemma~\ref{hs} below),
one easily derives the estimate
\begin{equation}
\norm{R(z)VR_0(z)}_{\Sch_1}\leq C\, \abs{\Im z}^{-1} \,,
\label{25}
\end{equation}
when $\Re z$ is positive and separated away from zero
(see Corollary~\ref{lma} below and note that the use of Lemma \ref{lap} can be avoided because of \eqref{10}).
On the other hand, Dynkin's theorem (see Theorem \ref{thm.dynkin} below) says that for $f\in B^1_{1,1}(\bbR)$ with 
support in $(0,\infty)$ 
there is an almost analytic continuation $\wt f$ with
support in $\{z: \Re z>0\}$ and  
\begin{equation}
\int_\bbC \abs{\overline\partial \wt f(z)}\, \frac{dx\, dy}{\abs{y}}<\infty, 
\quad z=x+iy.
\label{26}
\end{equation}
Putting together \eqref{24}, \eqref{25} and \eqref{26} yields
\begin{align*}
\norm{f(H)-f(H_0)}_{\Sch_1}
& \leq
\frac1\pi \int_\bbC \abs{\overline{\partial} \wt f(z)}\norm{R(z)VR_0(z)}_{\Sch_1} \,dx\, dy
\\
& \leq
C
\int_\bbC \abs{\overline{\partial} \wt f(z)}\frac{dx\, dy}{\abs{y}}<\infty \,,
\end{align*}
which yields Theorem~\ref{thm0} for $d=1,2,3$.
\end{proof}

We emphasize again that for a general pair of operators $H$, $H_0$ with $V=H-H_0\in\Sch_1$, 
one only has 
$$
\norm{R(z)VR_0(z)}_{\Sch_1}\leq \norm{V}_{\Sch_1}\, \abs{\Im z}^{-2} 
$$
instead of \eqref{25}, which leads to more restrictive assumptions on $f$, see Theorem~\ref{thm.p1}.

\subsection{Connection to the spectral shift function theory}

Implication \eqref{6} is intimately related to the spectral shift function theory (see, e.g., \cite{YaI,Ya}).
Let $\mathfrak M$ be the class of functions such that \eqref{6} holds for any self-adjoint operators $H$ and $H_0$ in a Hilbert space.
If $H-H_0\in\Sch_1$, M.~G.~Krein proved \cite{Krein}
that there is a real-valued function $\xi\in L^1(\bbR)$ such that for 
a suitable subclass of functions $f\in\mathfrak M$, the trace formula 
\begin{equation}
\Tr(f(H)-f(H_0))=\int_{-\infty}^\infty \xi(\lambda)f'(\lambda) d\lambda
\label{7}
\end{equation}
holds true. 
The function $\xi$ is called \emph{M.~G.~Krein's spectral shift function}. 
For any $\xi\in L^1(\bbR)$ there is a pair of operators $H_0$, $H$ such that 
$\xi$ is a spectral shift function for this pair.

The intuition coming from the spectral shift function theory allows one to 
interpret the above results as follows. 
Fix $f$; if one wants \eqref{7} to hold for \emph{any} 
self-adjoint operators $H$, $H_0$ with $H-H_0\in\Sch_1$, then the right side of \eqref{7}
must be well defined for \emph{any} $\xi\in L^1$.
Thus, necessarily we must have $f'\in L^\infty$. 
On the other hand, it is known that the spectral shift function
corresponding to the Schr\"odinger pair \eqref{1} with $V$ satisfying \eqref{2} for some $\rho>d$ is continuous on $(0,\infty)$; see, e.g., \cite[Theorem 9.1.20]{Ya}. Thus, in this case the right hand side of the 
trace formula \eqref{7} is well defined under the weaker assumption $f'\in L^1$. 

Similarly, under the assumption $V=H-H_0\in\Sch_2$ one can prove the 
existence of a function $\eta\in L^1(\bbR)$ such that 
$$
\Tr\left(f(H)-f(H_0)-\frac{d}{d\alpha}f(H_0+\alpha V)|_{\alpha=0}\right)
=
\int_{-\infty}^\infty \eta(\lambda)f''(\lambda)\,d\lambda \,;
$$
this was proven by Koplienko \cite{Koplienko} for a subclass of rational functions $f$
and by Peller \cite{Peller2} for $f\in B_{\infty,1}^2(\bbR)$. 
The function $\eta$ is called \emph{Koplienko's spectral shift function}; see also \cite{GPS} for some further information on this function. 

\subsection{Notation}
Throughout the paper, we use notation \eqref{23} for the resolvents of operators $H_0$ and $H$.
For $z\in\bbC$, we write $z=x+iy$ (we use the boldface $\x$ for the independent
variable in $\bbR^d$ when discussing the Schr\"odinger operator in $L^2(\bbR^d)$). 
For $p\geq1$, $\Sch_p$ is the Schatten class. 
The norm in any Banach space $X$ is denoted by $\norm{\cdot}_{X}$, and $\norm{\cdot}$ refers to
the operator norm.

\section{E.~M.~Dynkin's characterisation of Besov classes and V.~V.~Peller's trace class theorems}\label{sec.b}

\subsection{Besov classes}
For background information on Besov classes we refer, for example, to Triebel's book \cite{Triebel}.
Besov classes can be described as follows. For $t\in\bbR$ we define the operator $\Delta_t$ by
$$
(\Delta_tf)(\lambda)=f(\lambda+t)-f(\lambda),
$$
and let $\Delta_t^n$ be the powers of $\Delta_t$. A function $f\in L^p(\bbR)$ belongs 
to $B^s_{p,q}(\bbR)$, $s>0$, $1\leq p\leq \infty$, $1\leq q<\infty$, if 
$$
\int_\bbR \frac{\norm{\Delta_t^n f}_{L^p}^q}{\abs{t}^{1+sq}} dt<\infty, 
$$
where $n$ is an integer such that $n>s$. (The choice of $n$ does not make any difference.) In fact, we will only deal with classes $B^s_{1,1}$ and $B^s_{\infty,1}$.

We make use of E.~M.~Dynkin's characterisation of Besov
spaces \cite{Dynkin} in terms of pseudoanalytic continuation.
We will only be interested in compactly supported functions. 
For such functions, the results of \cite{Dynkin} can be expressed as follows:

\begin{theorem}\cite{Dynkin} \label{thm.dynkin}
Let $s>0$, $1\leq p\leq \infty$, $1\leq q< \infty$. 
For any compactly supported function $f\in B^s_{p,q}(\bbR)$, there is 
a (non-unique!) compactly supported function $\omega$ on $\bbC$ 
such that 
\begin{equation}
f(\lambda)=\frac1\pi \int_{\bbC}\omega(z)(\lambda-z)^{-1}dx\, dy, 
\quad 
\lambda\in\bbR, 
\quad 
z=x+iy,
\label{13}
\end{equation}
and 
\begin{equation}
\left(\int_\bbR 
\left(\int_\bbR \abs{\omega(x+iy)}^p \frac{dx}{\abs{y}^{p(s-1)}}\right)^{q/p}
\frac{dy}{\abs{y}}
\right)^{1/q}
<\infty.
\label{14}
\end{equation}
If $\supp f\subset[a,b]$, then for any $\eps>0$ the function $\omega$ can 
be chosen to be supported in the $\eps$-neighbourhood of $[a,b]$ in $\bbC$. 
\end{theorem}
In fact, the condition given in the theorem is necessary and sufficient for
the inclusion $f\in B^s_{p,q}(\bbR)$, and the Besov norm of $f$ is 
equivalent to the infimum of the expression \eqref{14} over all possible 
functions $\omega$. The function $\omega$ is usually obtained as
$$
\omega(z)=\overline{\partial} \wt f(z),
$$
where $\widetilde f$ is an almost analytic continuation of $f$. 
An almost analytic  continuation of $f$ can be constructed in several possible 
ways and so it is convenient not to fix the choice of $\omega$. 

For $p=q=1$ condition \eqref{14} becomes
\begin{equation}
f\in B^s_{1,1}(\bbR) 
\Leftrightarrow
\int_{\bbR^2} \abs{\omega(x+iy)}\,\frac{dx\, dy}{\abs{y}^s}<\infty,
\label{15}
\end{equation}
and for $p=\infty$, $q=1$ we get
\begin{equation}
f\in B^s_{\infty,1}(\bbR) 
\Leftrightarrow
\int_{\bbR} \sup_{x}\abs{\omega(x+iy)}\,\frac{dy}{\abs{y}^s}<\infty.
\label{16}
\end{equation}

\subsection{Peller's trace class theorems}
To demonstrate the effectiveness of Dynkin's characterization, below we give short proofs of the following 
two theorems of Peller mentioned in the introduction.

\begin{theorem}\cite{Peller}\label{thm.p1}
Let $f\in B^1_{\infty,1}(\bbR)$.
Then the implication \eqref{6} holds true and, for some absolute constant $C$, 
one has
$$
\left\| f(H)-f(H_0)\right\|_{\Sch_1}\leq C\norm{f}_{B^1_{\infty,1}}\norm{H-H_0}_{\Sch_1} \,.
$$
\end{theorem}

\begin{theorem}\cite{Peller2}\label{thm.p2}
Let $f\in B^2_{\infty,1}(\bbR)$. Then the implication \eqref{12} holds true, 
where the derivative exists in the operator norm, and, for some absolute constant $C$, 
one has
$$
\left\| f(H)-f(H_0)- \frac{d}{d\alpha}f((1-\alpha)H_0-\alpha H)|_{\alpha=0} \right\|_{\Sch_1}\leq C\norm{f}_{B^2_{\infty,1}}\norm{H-H_0}_{\Sch_2} \,.
$$
\end{theorem}

As elsewhere in the paper, we assume that $f$ is compactly supported, although in fact Peller's original results do not require this.

\begin{proof}[Proof of Theorem~\ref{thm.p1}]
We use representation \eqref{13}, where $\omega$ satisfies \eqref{16} with $s=1$. 
As in \eqref{24}, by the resolvent identity, we can write
$$
f(H)-f(H_0)=-\frac1\pi \int_{\bbC} \omega(z)R_0(z)VR(z)\,dx\, dy\,.
$$
We have
\begin{align*}
\norm{R_0(z)VR(z)}_{\Sch_1}
& \leq
\norm{R_0(z)\abs{V}^{1/2}}_{\Sch_2} \norm{\abs{V}^{1/2}R(z)}_{\Sch_2}
\\
& \leq
\frac12 \norm{R_0(z)\abs{V}^{1/2}}_{\Sch_2}^2+
\frac12 \norm{R(z)\abs{V}^{1/2}}_{\Sch_2}^2.
\end{align*}
Write the spectral representation of the trace class operator $\abs{V}$ as
$$
\abs{V}=\sum_{n=1}^\infty v_n (\cdot,\psi_n)\psi_n,
\quad
\norm{\psi_n}=1, 
\quad 
v_n\geq0, 
\quad 
\sum_{n=1}^\infty v_n=\norm{V}_{\Sch_1}<\infty. 
$$
We have
\begin{align*}
\norm{R_0(z)\abs{V}^{1/2}}_{\Sch_2}^2
& =
\norm{R_0(z)\abs{V}R_0(\overline{z})}_{\Sch_1}
\\
& =
\sum_{n=1}^\infty v_n 
\norm{R_0(z) \psi_n}^2 \\
& =
\sum_{n=1}^\infty v_n 
\int_{\bbR}\frac{d\mu_{\psi_n}(t)}{(t-x)^2+y^2}\,,
\end{align*}
where $\mu_{\psi_n}$ is the spectral measure  of $H_0$ corresponding to $\psi_n$:
$$
\mu_{\psi_n}(\delta)=(\1_\delta(H_0)\psi_n,\psi_n) \,.
$$
We obtain
\begin{align*}
\int_{\bbC} \abs{\omega(z)}
\norm{R_0(z)\abs{V}^{1/2}}_{\Sch_2}^2
dx\, dy
& =
\sum_{n=1}^\infty v_n 
\int_{\bbC} \abs{\omega(z)}
\int_{\bbR}\frac{d\mu_{\psi_n}(t)}{(t-x)^2+y^2}\,dx\, dy
\\
& \leq \pi
\sum_{n=1}^\infty v_n 
\int_{\bbR}
\sup_x \abs{\omega(z)} \,\frac{dy}{\abs{y}} \int_{\bbR} d\mu_{\psi_n}(t)
\\
& \leq \pi
\norm{V}_{\Sch_1}\int_\bbR\sup_x \abs{\omega(z)}\,\frac{dy}{\abs{y}}
<\infty \,.
\end{align*}
Of course, in the same way we get an estimate for the integral involving 
$\norm{R(z)\abs{V}^{1/2}}_{\Sch_2}^2$.
Thus, we obtain
$$
\norm{f(H)-f(H_0)}_{\Sch_1}
\leq 
C(f)\norm{V}_{\Sch_1} \,,
$$
where, according to Theorem \ref{thm.dynkin} and the remark thereafter,
$$
C(f)= \int_\bbR\sup_x \abs{\omega(z)}\,\frac{dy}{\abs{y}}\asymp \norm{f}_{B^1_{\infty,1}}.
$$
This concludes the proof of Theorem \ref{thm.p1}.
\end{proof}

\begin{proof}[Proof of Theorem~\ref{thm.p2}]
Let $f$ be represented as in \eqref{13}, where $\omega$ satisfies \eqref{16} with $s=2$. 
By a direct calculation, we have
\begin{equation}
\frac{d}{d\alpha}(H_0+\alpha V-z)^{-1}|_{\alpha=0}=-R_0(z)VR_0(z), 
\label{17}
\end{equation}
and
\begin{align}
R(z)-R_0(z)-\frac{d}{d\alpha}(H_0+\alpha V-z)^{-1}|_{\alpha=0}
& = -R(z)VR_0(z)+R_0(z)VR_0(z)
\notag \\
& =R_0(z)VR(z)VR_0(z).
\label{18}
\end{align}
From \eqref{13}, \eqref{16} and \eqref{17}  and the estimate 
$$
\norm{(H_0+\alpha V-z)^{-1}VR_0(z)}\leq \norm{V}\abs{\Im z}^{-2}\,,
$$
it is straightforward to see that 
$$
\frac{d}{d\alpha}f(H_0+\alpha V)|_{\alpha=0}
=
-\frac1\pi \int_{\bbC}\omega(z) R_0(z)VR_0(z) \,dx\, dy\,, 
$$
where the derivative exists in the operator norm and the 
integral converges absolutely in the operator norm. 
Next, by \eqref{18}, 
$$
f(H)-f(H_0)-\frac{d}{d\alpha}f(H_0+\alpha V)|_{\alpha=0}
=
\frac1\pi\int_{\bbC}\omega(z)R_0(z)VR(z)VR_0(z) \,dx\, dy\,.
$$
Finally, 
\begin{align*}
\norm{R_0(z)VR(z)VR_0(z)}_{\Sch_1}
& \leq
\norm{R_0(z)V}_{\Sch_2}\norm{R(z)}\norm{VR_0(z)}_{\Sch_2}
\\
& \leq
\frac1{\abs{\Im z}}\norm{R_0(z)V}_{\Sch_2}^2 \,,
\end{align*}
and the rest of the proof proceeds exactly as in Theorem~\ref{thm.p1}.
\end{proof}

\section{Proof of Theorems \ref{main} and \ref{thm2}}

\subsection{Some preliminary bounds}

As explained in the introduction, the proof of Theorems \ref{main} and \ref{thm2} 
relies on a combination of ideas from trace class (Lemmas \ref{hs} and \ref{diff}) 
and from smooth scattering theory (Lemma \ref{lap}). In this subsection we collect the necessary bounds.
Throughout the rest of the paper, $H_0=-\Delta$, $H=-\Delta+V$,  and $R_0(z)$, $R(z)$ are the corresponding
resolvents. 

The following two lemmas are standard in trace class scattering theory. 

\begin{lemma}\label{hs}
Let $d\geq 1$, $\kappa>d/4-1$, $E>0$ and let $\delta\subset(0,\infty)$ be a compact interval. 
Then there is a constant $C>0$ such that for all $W\in L^2(\R^d)$ and for all $z$ with $\Re z\in\delta$ and $\Im z\neq 0$ we have
$$
\left\| W R_0(z) R_0(-E)^\kappa \right\|^2_{\mathbf S_2} \leq C |\im z|^{-1} \|W\|_2^2 \,.
$$
\end{lemma}

\begin{proof}
The left side is equal to
$$
(2\pi)^{-d} \int_{\R^d} |W(\x)|^2 \,d\x \int_{\R^d} \frac{d\p}{\left|\left|\p\right|^2-z\right|^2 \left(\left|\p\right|^2 +E\right)^{2\kappa}} \,. 
$$
By splitting the integral into the region where $\re z/2 \leq |\p|^2 \leq 2 \re z$ and its complement, we easily obtain the bound of the lemma.
\end{proof}

\begin{lemma}\label{diff}
Let $d\geq 4$ and let $k$ be an integer with $k>d/2-1$. Assume that $V\in \ell^1(L^2)$ is form-bounded with respect to $-\Delta$ with form bound $<1$. Then for all sufficiently large $E>0$,
$$
R(-E)^k - R_0(-E)^k \in \mathbf S_1 \,.
$$
\end{lemma}

\begin{proof}
We shall use a result of Reed and Simon \cite{ReSi} (see also \cite[Thm. XI.12]{ReSi3}), 
closely related to an earlier result of Yafaev \cite{Ya0}. 
According to this result our assertion follows from the fact that
\begin{equation}
R_0(-E)^{1/2} V R_0(-E)^{k+1/2} \in \mathbf S_1 \,.
\label{8b}
\end{equation}
The inclusion \eqref{8b} follows by interpolation from 
$$
R_0(-E)^{k+1} V\in\Sch_1, 
\quad 
V R_0(-E)^{k+1} \in\Sch_1.
$$
Finally, the last two inclusions were proven by Birman and Solomyak
(see \cite{BS2} or \cite[Thm. 4.5]{Si}). 
\end{proof}

The next assertion is a form of the limiting absorption principle. 

\begin{lemma}\label{lap}
Let $d\geq 1$ and let
\begin{align*}
& p=1 & \text{if}\ d=1 \,,\\
& 1<p\leq 3/2 & \text{if}\ d=2 \,,\\
& d/2\leq p\leq (d+1)/2 & \text{if}\ d\geq 3 \,.
\end{align*}
Assume that $V\in L^p(\R^d)$. 
Then for any compact interval $\delta\subset(0,\infty)$ there is a constant 
$C>0$ such that for any $z\in\C$ with $\re z\in \delta$ and $\im z\neq 0$,
$$
\left\| \sqrt{|V|} R(z) \sqrt{|V|} \right\| \leq C \,.
$$
\end{lemma}

The proof of the lemma only under $L^p$ conditions on $V$ is not completely standard and, for $d\geq 2$, relies on some results in harmonic analysis. It is essentially contained in the papers \cite{IoSc,KoTa}. We defer a discussion of the proof to Subsection \ref{sec:lap}.

For the moment we note that, if the $L^p$ condition on $V$ is replaced by the pointwise condition \eqref{2} with $\rho>1$, then the bound of Lemma \ref{lap} follows directly from the classical limiting absorption principle \eqref{10}.

We now combine the bounds from Lemmas \ref{hs} and \ref{lap} and obtain

\begin{corollary}\label{lma}
Let $d\geq 1$ and 
\begin{align*}
& V \in L^1(\R) & \text{if}\ d=1 \,, \\
& V \in L^1(\R^2)\cap L^p(\R^2) \ \text{for some}\ p>1 & \text{if}\ d=2\,,\\
& V\in L^1(\R^d)\cap L^{d/2}(\R^d) & \text{if}\ d\geq 3 \,.
\end{align*}
For every compact interval $\delta\subset(0,\infty)$, every $\kappa>d/4-1$ and every $E>0$ there is a constant $C>0$ such that for all $z$ with $\Re z\in\delta$ and $\abs{\Im z}\leq1$, $\Im z\not=0$, we have
$$
\norm{R_0(-E)^{\kappa} (R(z)-R_0(z)) R_0(-E)^{\kappa}}_{\Sch_1}
\leq
C/\abs{\Im z} \,.
$$
\end{corollary}

Note that if $d\leq 3$, then we can choose $\kappa=0$.

\begin{proof}
It is well-known that under the conditions of the corollary, $V$ is infinitesimally form-bounded with respect to $-\Delta$ and therefore $H$ can be defined via a quadratic form with form domain $H^1(\R^d)$. Iterating the resolvent identity, we obtain:
\begin{align}
R(z)-R_0(z)
& =
-R(z)VR_0(z)
=
-R_0(z)VR_0(z)
+
R_0(z)VR(z)VR_0(z)
\notag \\
& =
R_0(z)\sqrt{|V|} (1+ \sqrt{V} R(z)\sqrt{|V|}) \sqrt V R_0(z) \,,
\label{9}
\end{align}
where we used the notation $\sqrt{V} = (\sgn V) \sqrt{|V|}$. By Lemmas \ref{hs} and \ref{lap} we obtain
\begin{align*}
& \norm{R_0(-E)^\kappa (R(z)-R_0(z)) R_0(-E)^\kappa }_{\Sch_1} \\
& \quad \leq
\norm{R_0(-E)^\kappa R_0(z)\sqrt{|V|}}_{\Sch_2}^2
\norm{1+\sqrt{V}R(z)\sqrt{|V|}} \\
& \quad \leq
C/\abs{\Im z} \,,
\end{align*}
as claimed.
\end{proof}


\subsection{Proof of Theorems \ref{main} and \ref{thm2}}

We are now in position to prove our main results.

\begin{proof}[Proof of Theorem~\ref{main}]
First we consider the case $d\leq 3$.
As in the proof of Theorem~\ref{thm.p1}, we get 
$$
f(H)-f(H_0)=\frac1\pi \int_\bbC \omega(z) (R(z)-R_0(z))\,dx\, dy, 
$$
where $\omega$ satisfies \eqref{15} with $s=1$. 
We can also ensure that $\supp \omega\subset \{z: \Re z\in \delta\}$ 
for some compact interval $\delta\subset (0,\infty)$. 
Then, using Corollary~\ref{lma} with $\kappa=0$, we obtain
\begin{align*}
\norm{f(H)-f(H_0)}_{\Sch_1}
& \leq
\frac1\pi
\int_{\bbC}
\abs{\omega(z)} \norm{R(z)-R_0(z)}_{\Sch_1}dx\, dy
\\
& \leq
\frac1\pi C  \int_{\bbC}
\abs{\omega(z)}\frac1{\abs{\Im z}}dx\, dy<\infty, 
\end{align*}
as required. 

Next, consider the case of dimensions $d\geq4$. 
Let $k$ be an integer such that $k>d/2-1$. (Note that $k\geq 2$ since $d\geq 4$.) 
Then, by Lemma \ref{diff},
\begin{equation}
R(-E)^k- R_0(-E)^k\in\Sch_1
\label{21}
\end{equation}
for all sufficiently large $E>0$. 
Fix such an $E$ and let $g(\lambda)=(\lambda+E)^{2k} f(\lambda)$. 
Clearly, $g\in B_{1,1}^1(\bbR)$ and $g$ is compactly supported.
Thus, we can represent $g$ in the same form as $f$, 
$$
g(\lambda)=\frac1\pi \int_{\bbC}\omega_g(z)(\lambda-z)^{-1} \,dx\, dy
$$
with 
$$
\int_\bbC \frac{\abs{\omega_g(z)}}{\abs{\Im z}}dx\, dy<\infty.
$$
We have 
\begin{align*}
f(H)-f(H_0)
& =
R(-E)^k g(H)R(-E)^k -R_0(-E)^k g(H_0)R_0(-E)^k
\\
& =\left(R(-E)^k-R_0(-E)^k\right)g(H)R(-E)^k
\\
& \quad +
R_0(-E)^k g(H)\left( R(-E)^k-R_0(-E)^k\right)
\\
& \quad +
R_0(-E)^k\left(g(H)-g(H_0)\right)R_0(-E)^k \,. 
\end{align*}
The first two terms in the right side are trace class operators by \eqref{21}. For the third term, we have
\begin{multline}
R_0(-E)^k\left(g(H)-g(H_0)\right)R_0(-E)^k
\\
=
\frac1\pi 
\int_\bbC \omega_g(z) R_0(-E)^k\left(R(z)-R_0(z)\right)R_0(-E)^k\, dx\, dy, 
\label{28}
\end{multline}
and so again using Corollary~\ref{lma} (with $\kappa=k$), we get the required result. 

Finally, let us prove the statement concerning the function $f_{\gamma,a}$.
By multiplying  $f_{\gamma,a}$ by suitable cutoff functions, it is easy to represent it as
$f_{\gamma,a}=f_0+f$, where $f_0$ is compactly supported on $(0,\infty)$ and belongs to $B^1_{1,1}(\bbR)$, 
and $f$ is infinitely smooth and vanishes for $\lambda>a$. Since only the values of $f$ on 
the spectra of $H_0$, $H$ 
 are relevant, we may assume that $f\in C_0^\infty(\bbR)$. Thus, it remains to prove that
$$
f(H)-f(H_0)\in\Sch_1,  \quad f\in C_0^\infty(\bbR).
$$
This statement is well known but for completeness let us indicate the proof by the same method as above.
A function $f\in C_0^\infty(\bbR)$ can be represented as in \eqref{13}, where $\omega$ satisfies \eqref{15} 
with any $s>0$; for us $s=3$ suffices. 
Next, we repeat the proof of the theorem for the case $d\geq 4$, but instead of applying Corollary \ref{lma} to \eqref{28} we estimate as follows, using the resolvent identity in the form \eqref{9}:
\begin{align*}
& \left\| R_0(-E)^k\left(R(z)-R_0(z)\right)R_0(-E)^k \right\| \\
& \quad \leq \norm{R_0(-E)^\kappa R_0(z)\sqrt{|V|}}_{\Sch_2}^2
\left( 1+ \| \sqrt{|V|} R(-E)^{1/2}\|^2 \|(H+E)R(z)\| \right)
\end{align*}
The first factor on the right side can be estimated by means of Lemma \ref{hs}. Since $V$ is infinitesimally form-bounded with respect to $H_0$, it is also infinitesimally form-bounded with respect to $H$ and therefore $\| \sqrt{|V|} R(-E)^{1/2}\|<\infty$. Finally,
$$
\|(H+E)R(z)\| \leq C \left(|\Im z|^{-1} + 1\right) \,.
$$
This implies that 
\begin{align*}
& \left\| R_0(-E)^k\left(R(z)-R_0(z)\right)R_0(-E)^k \right\| \leq C'|\Im z|^{-2} \left(|\Im z|^{-1} + 1\right) \,.
\end{align*}
Again, combining this with \eqref{28}, we obtain the required result. (The term $+1$ in the last bound is irrelevant with respect to $|\Im z|^{-1}$ since, as stated in Theorem \ref{thm.dynkin}, $\omega$ may be chosen to have compact support.)
\end{proof}

\begin{proof}[Proof of Theorem~\ref{thm2}]
We follow the proof of Theorem~\ref{thm.p2}. 
Let $f$ be represented as in \eqref{13} with $\omega$ satisfying
\eqref{15} with $s=2$ and 
$\supp \omega\subset\{z: \Re z\in\delta\}$, $\delta\subset(0,\infty)$.
As in Theorem~\ref{thm.p2}, 
we get the representation
$$
f(H)-f(H_0)-\frac{d}{d\alpha}f(H_0+\alpha V)|_{\alpha=0}
\\
=
-\frac1\pi\int_{\bbC}\omega(z)R_0(z)VR(z)VR_0(z) \,dx\, dy \,.
$$
According to Lemma \ref{hs},
\begin{equation}
\norm{R_0(z)VR(z)VR_0(z)}_{\Sch_1}
\leq
\abs{\Im z}^{-1} \norm{R_0(z)V}_{\Sch_2}^2
\leq C(\delta)\abs{\Im z}^{-2}.
\label{29}
\end{equation}
Now combining this with \eqref{15}, we obtain the required statement.

To prove the statement concerning the function  $f_{\gamma,a}$, 
as in the proof of Theorem~\ref{main} we represent $f_{\gamma,a}=f_0+f$
and reduce the proof to the inclusion
$$
f(H)-f(H_0)-\frac{d}{d\alpha} f(H_0+\alpha V)|_{\alpha=0}\in\Sch_1, 
\quad
f\in C_0^\infty(\bbR).
$$
This is proven in the same way by using the estimate
\begin{align*}
\norm{R_0(z)VR(z)VR_0(z)}_{\Sch_1}
& =
\norm{R_0(z)(H_0+I)R_0(-1)VR(z)VR_0(-1)(H_0+I)R_0(z)}_{\Sch_1}
\\
& \leq
\norm{R_0(z)(H_0+I)}^2\norm{R_0(-1)V}_{\Sch_2}^2\norm{R(z)}
\\
& \leq
C\abs{\Im z}^{-1} \left(|\Im z|^{-1} + 1\right)^2
\end{align*}
instead of \eqref{29}. This completes the proof of Theorem \ref{thm2}.
\end{proof}


\subsection{Proof of Lemma \ref{lap}}\label{sec:lap}

In this subsection we shall prove

\begin{lemma}\label{key}
Let $d\geq 1$ and let
\begin{align*}
& p=1 & \text{if}\ d=1 \,,\\
& 1<p\leq 3/2 & \text{if}\ d=2 \,,\\
& d/2\leq p\leq (d+1)/2 & \text{if}\ d\geq 3 \,.
\end{align*}
Assume that $V\in L^p(\R^d)$. Then for any compact interval $\delta\subset(0,\infty)$ there is a constant $C>0$ such that for any $z\in\C$ with $\re z\in\delta$ and $0<|\im z|\leq 1$,
$$
\left\| R(z) \right\|_{L^r \to L^{r'}} \leq C \,,
$$
where $r=2p/(p+1)$ and $r'$ is the dual exponent, $\frac1r+\frac1{r'}=1$.
\end{lemma}

We claim that this lemma implies Lemma \ref{lap}. 
Indeed, if $V\in L^p$ then H\"older's inequality implies that multiplication by 
$\sqrt{|V|}$ is a bounded operator from $L^2$ to $L^r$ with $r=2p/(p+1)$. 
By duality, multiplication by $\sqrt{|V|}$ is a bounded operator from $L^{r'}$ to $L^2$. 
Thus, the bound from Lemma \ref{key} implies the bound in Lemma \ref{lap} for $|\im z|\leq 1$. 
The fact that the bound in Lemma \ref{key} holds also for $|\im z|>1$ follows from the form-boundedness 
of $V$ with respect to $H$, which implies that
$$
\left\| \sqrt{|V|} R(z) \sqrt{|V|} \right\| \leq \left\| \sqrt{|V|} R(-E)^{1/2}\right\|^2  \left\| (H+E) R(z)\right\|
\leq C
$$
for some large $E>0$ and all $|\im z|\geq 1$.

For the proof of Lemma \ref{key} we distinguish the cases $d=1$ and $d\geq 2$.

\begin{proof}[Proof of Lemma \ref{key} for $d=1$]
We write $z=k^2$ with $\im k \geq 0$. Let $\theta_\pm(\cdot,k)$ be the Jost solutions, that is, solutions of $-\theta'' + V\theta = k^2 \theta$ on $\R$ with
$$
\theta_\pm(x,k) \sim e^{\pm ikx}
\qquad\text{as}\ x\to\pm\infty \,.
$$
In terms of these functions the resolvent kernel is given by
$$
R(k^2)(x,x') = \frac{\theta_+(x_>,k)\theta_-(x_<,k)}{w(k)} \,,
$$
where $x_>=\max\{x,x'\}$, $x_<=\min\{x,x'\}$ and
$$
w(k) = \theta_-'(x,k) \theta_+(x,k) - \theta_-(x,k) \theta_+'(x,k) \,
$$
is the Wronskian.
The quantity we need to bound is
$$
\| R(k^2) \|_{L^1\to L^\infty} = \sup_{x,x'} |R(k^2)(x,x')| \,.
$$
It is well known (see e.g. \cite{Ya}) that the functions $e^{\mp ikx} \theta_\pm(x,k)$ are 
uniformly bounded for $x\in\bbR$ and $k$ bounded away from zero; 
also the Wronskian is a continuous function  without
zeros on $(0,\infty)$. This gives the required result for $d=1$.
\end{proof}

The case $d\geq 2$ is more complicated and relies on harmonic analysis results in \cite{IoSc} and \cite{KoTa}. We also note that the bound of the lemma in the case $V\equiv 0$ is due to \cite{KeRuSo} (see also \cite{Fr} for the case $d=2$).

\begin{proof}
[Proof of Lemma \ref{key} for $d\geq 2$]
We aim at applying the results of \cite{IoSc}. Let us show that our operator fits in their framework. (This is claimed without proof in their paper, but for the sake of completeness we provide a short argument.) Let $q_0=p$ if $d=2$ and $q_0=d/2$ if $d\geq 3$ and consider
$$
M(\x) = \left( \int_{|\x'-\x|\leq 1/2} |V(\x')|^{q_0} \,d\x' \right)^{1/q_0} \,.  
$$
Then, by H\"older, since $p\geq q_0$,
$$
M(\x) \leq \left( \omega_d 2^{-d} \right)^{(p-q_0)/(p q_0)} \left( \int_{|\x'-\x|\leq 1/2} |V(\x')|^p \,d\x' \right)^{1/p} \,,
$$
and therefore, by Minkowski, since $p\leq (d+1)/2$,
$$
\int_{\R^d} M(\x)^{(d+1)/2} \,d\x \leq \left( \omega_d 2^{-d} \right)^{1+(p-q_0)(d+1)/(2pq_0)} \|V\|_{L^p}^{(d+1)/2} <\infty \,.
$$
According to \cite[Prop. 1.4]{IoSc} this means that $V$ is an admissible perturbation in the sense of \cite{IoSc}.

Therefore, \cite[Thm. 1.3]{IoSc} states that for a certain Banach space $X$ 
and for any compact interval $\delta\subset(0,\infty)$, which does not contain eigenvalues of $H$,
\begin{equation}
\label{eq:is}
\sup_{\re z\in \delta,\,\,   0<|\im z|\leq 1} \| R(z) \|_{X\to X'} <\infty \,.
\end{equation}
The Banach space $X$ satisfies $W^{-1/(d+1),2(d+1)/(d+3)}(\R^d)\subset X$ (continuously). (Here we use the notation $W^{s,p}(\R^d)$ for the Sobolev space of order $s$ with integrability index $p$; see \cite{Triebel}.) We now apply the Sobolev embedding theorem which implies that there is a constant $C_{d,r}$ such that
\begin{equation}
\label{eq:sob}
\| u \|_{L^{r'}} \leq C_{d,r} \|u\|_{W^{1/(d+1),2(d+1)/(d-1)}} \,.
\end{equation}
(Here we use the fact that $p>1$ in $d=2$, which implies that $r'<\infty$, and that $p\geq d/2$ in $d\geq 3$, which implies that $r'\leq 2d/(d-2)$.) By duality, \eqref{eq:sob} yields
\begin{equation}
\label{eq:sobdual}
\|u\|_{W^{-1/(d+1),2(d+1)/(d+3)}} \leq C_{d,r} \|u\|_{L^r}
\end{equation}
and, therefore, \eqref{eq:is} implies
$$
\sup_{\re z\in \delta, 0<|\im z|\leq 1} \| R(z) \|_{L^r\to L^{r'}} <\infty
$$
for any compact interval $\delta\subset(0,\infty)$, which does not contain eigenvalues of $H$.

We now apply the main result of \cite{KoTa}, which states that the operator $H$ has no positive eigenvalues if $V\in L^p(\R^d)$ with $p$ as in the statement of the lemma. Again, this is stated in their paper, but the general condition in the paper is not explicitly verified, so we give a quick argument. Consider the norm
$$
\|U\|_X = \sup_{\phi\in C_0^\infty} \frac{\|U\phi\|_{W^{-s,p}}}{\|\phi\|_{W^{s,p'}}}
$$
where $s=1/3-\epsilon$ if $d=2$ and $s=1/(d+1)$ if $d\geq 3$ and where $p$ is as in the statement of the lemma. The assumption of \cite{KoTa} is satisfied if
$$
\lim_{j\to\infty} \|V \1_{\{2^{j-1}<|x|<2^{j+1}\}} \|_X = 0 \,.
$$
To verify this condition we use again \eqref{eq:sob} and \eqref{eq:sobdual} which, by H\"older, imply that
$$
\|U\|_X \leq C_{d,r}^2 \sup_{\phi\in C_0^\infty} \frac{\|U\phi\|_{L^r}}{\|\phi\|_{L^{r'}}} \leq C_{d,r}^2 \|U\|_{L^p} \,.
$$
(If $d=2$, we have to replace \eqref{eq:sob} by $\|u\|_{L^{r'}} \leq \|u\|_{W^s_{6/5}}$, which holds for $\epsilon$ small enough depending on $r$.) Thus, the assumption $V\in L^p(\R^d)$ implies that
$$
\lim_{j\to\infty} \|V \1_{\{2^{j-1}<|\x|<2^{j+1}\}} \|_X \leq C_{d,r}^2 \lim_{j\to\infty} \|V \1_{\{2^{j-1}<|\x|<2^{j+1}\}} \|_{L^p} = 0 \,,
$$
as required. This concludes the proof of Lemma \ref{key} for $d\geq 2$.
\end{proof}


\section{Proof of Theorem~\ref{lt} and Corollary~\ref{lt_cor}}

In order to deduce Theorem~\ref{lt} from the results of \cite{FLLS} we need the following abstract lemma. 
We write $P^\bot = 1-P$ if $P$ is a projection.

\begin{lemma}\label{tracelemma}
Let $A$ and $B$ be self-adjoint operators and denote $P=\1_{\{A<0\}}$ and $Q = \1_{\{B<0\}}$. Assume that $P(B-A)P\in\mathbf S_1$. Then $P(B_--A_-)P\in\mathbf S_1$, $P^\bot(B_--A_-)P^\bot \in\mathbf S_1$ iff $(Q-P) |B|^{1/2}\in\mathbf S_2$, and in this case
$$
\tr P(B_-- A_-)P + \tr P^\bot(B_-- A_-)P^\bot + \tr P(B-A)P = \left\| (Q-P) |B|^{1/2} \right\|_{\mathbf S_2}^2 \,.
$$
\end{lemma}

\begin{proof}
We first observe that $(Q-P)^2 = QP^\bot Q + Q^\bot P Q^\bot$. Thus,
\begin{align*}
\left\| (Q-P) |B|^{1/2} \right\|_{\mathbf S_2}^2
& = \tr |B|^{1/2} \left( QP^\bot Q + Q^\bot P Q^\bot \right) |B|^{1/2} \\
& = \tr \left( B_-^{1/2} P^\bot B_-^{1/2} + B_+^{1/2} P B_+^{1/2} \right) \,.
\end{align*}
Since both $B_-^{1/2}P^\bot B_-^{1/2}$ and $B_+^{1/2} P B_+^{1/2}$ are non-negative operators, we have $(Q-P) |B|^{1/2}\in\mathbf S_2$ iff $B_-^{1/2} P^\bot B_-^{1/2}\in\mathbf S_1$ and $B_+^{1/2} P B_+^{1/2}\in\mathbf S_1$. This is equivalent to having $P^\bot B_- P^\bot\in\mathbf S_1$ and $P B_+ P\in\mathbf S_1$, and in this case
\begin{equation}
\label{eq:cyclicity}
\tr B_-^{1/2} P^\bot B_-^{1/2} = \tr P^\bot B_- P^\bot \,,
\qquad
\tr B_+^{1/2} P B_+^{1/2} = \tr P B_+ P \,.
\end{equation}
Now we note that
$$
P^\bot B_- P^\bot = P^\bot (B_--A_-) P^\bot
$$
and that
\begin{align*}
P B_+ P = P \left(B_- + B \right) P = P \left(B_- + A + \left(B-A\right) \right) P = P \left(B_- - A_- + \left(B-A\right) \right) P \,.
\end{align*}
This, together with \eqref{eq:cyclicity} proves the lemma.
\end{proof}

\begin{proof}[Proof of Theorem~\ref{lt}]
We want to apply Lemma \ref{tracelemma} with $A=-\Delta-\mu$ and $B=-\Delta+V-\mu$. Define $P$ and $Q$ as in that lemma. 
By Theorem \ref{main}, we have $B_--A_-\in\mathbf S_1$ under our hypothesis. 
Thus, in particular, $P(B_--A_-)P\in\mathbf S_1$ and $P^\bot (B_--A_-)P^\bot\in\mathbf S_1$ and
$$
\tr P(B_-- A_-)P + \tr P^\bot(B_-- A_-)P^\bot = \tr\left( B_- - A_-\right) \,.
$$ 
Moreover, since $V\in L^1(\R^d)$ we have $\sqrt{|V|} P \in\mathbf S_2$ and therefore $P(B-A)P\in\mathbf S_1$ with
$$
\tr P(B-A)P = L_{0,d}^{\mathrm sc}\ \mu_+^{d/2} \int_{\R^d} V \,dx \,.
$$
Thus, we deduce from Lemma \ref{tracelemma} that
\begin{align*}
\tr \left( \left(-\Delta+V-\mu\right)_- - \left(-\Delta-\mu\right)_- \right) + L_{0,d}^{\mathrm sc}\ \mu_+^{d/2} \int_{\R^d} V \,dx  = \left\| (Q-P) |B|^{1/2} \right\|_{\mathbf S_2}^2 \,.
\end{align*}
Obviously, the right side is non-negative. 
Moreover, the bound on the right side from \cite{FLLS} yields Theorem \ref{lt}.
\end{proof}

\begin{proof}[Proof of Corollary~\ref{lt_cor}]
This is easy to obtain by a modification of an argument of Aizenman--Lieb \cite{AL}. 
Indeed, by the functional calculus
$$
(-\Delta + V - \mu)_-^\gamma = \gamma(\gamma-1) \int_0^\infty (-\Delta + V - \mu+\tau)_- \tau^{\gamma-2}\,d\tau
$$
and similarly for $(-\Delta-\mu)_-^\gamma$. Theorem \ref{main}, together with the assumptions on $V$, allows to interchange taking the trace and integrating with respect to $\tau$, and we obtain
\begin{align*}
& \tr\left( \left(-\Delta+V-\mu\right)_-^\gamma - \left(-\Delta -\mu \right)_-^\gamma \right) \\
&  \qquad= \gamma (\gamma-1) \int_0^\infty \tr\left( \left(-\Delta+V-\mu+\tau \right)_- - \left(-\Delta -\mu+\tau \right)_- \right) \tau^{\gamma-2}\,d\tau \,. 
\end{align*}
Moreover, a simple computation shows that
$$
L_{\gamma-1,d}^{\mathrm{sc}} \ \mu_+^{\gamma+d/2-1}
= \gamma(\gamma-1) \int_0^\infty L_{0,d}^{\mathrm{sc}}\, (\mu -\tau)_+^{d/2} \tau^{\gamma-2}\,d\tau \,.
$$
Thus,
\begin{align*}
& \tr\left( \left(-\Delta+V-\mu\right)_-^\gamma - \left(-\Delta -\mu \right)_-^\gamma \right) 
+ L_{\gamma-1,d}^{\mathrm{sc}} \ \mu_+^{\gamma+d/2-1} \int_{\R^d} V \,dx \\
& \qquad = \gamma(\gamma-1) \int_0^\infty \left(
\tr\left( \left(-\Delta+V-\mu+\tau \right)_- - \left(-\Delta -\mu +\tau \right)_- \right) \phantom{\int} \right. \\
& \qquad\qquad \qquad\qquad\qquad  \left. + L_{0,d}^{\mathrm{sc}} \ (\mu-\tau)_+^{d/2} \int_{\R^d} V \,dx \right) \tau^{\gamma-2} \,d\tau
\end{align*}
We now apply Theorem~\ref{lt} for every $\tau>0$. Observing that
\begin{align*}
& \int_0^\infty \int_{\R^d} \left( \left(V-\mu+\tau\right)_-^{1+d/2} - (\mu-\tau)_+^{1+d/2} + \left(1+\frac d2\right) (\mu-\tau)_+^{d/2} V \right) dx\,\tau^{\gamma-2} \,d\tau \\
& \qquad = \frac{\Gamma(\frac d2+2)\,\Gamma(\gamma-1)}{\Gamma(\gamma+\frac d2+1)} \int_{\R^d} \left( \left(V-\mu\right)_-^{\gamma+\frac d2} - \mu_+^{\gamma+\frac d2} + \left(\gamma+\frac d2\right) \mu_+^{\gamma+d/2-1} V \right) dx \,,
\end{align*}
we obtain the claimed inequality.
\end{proof}

\subsection*{Acknowledgments} The first author is grateful to M. Lewin, J. Sabin and D. Yafaev for useful discussions. Financial support from the U.S.~National Science Foundation through grant PHY-1347399 (R. F.) is acknowledged.


\bibliographystyle{amsalpha}

\end{document}